\documentclass[aps,reprint,superscriptaddress,pra]{revtex4-2}
\usepackage{amsmath, amssymb, amsthm}
\usepackage{booktabs}
\usepackage{float}
\usepackage{microtype}
\usepackage{parskip}
\usepackage{array}
\usepackage[colorlinks=true, linkcolor=blue, citecolor=blue, urlcolor=blue]{hyperref}
\usepackage{tikz}
\usepackage{graphicx}
\usepackage[section]{placeins}
\usetikzlibrary{positioning}

\newtheorem{definition}{Definition}

\newtheorem{proposition}{Proposition}

\begin{document}

\title{\textbf{Financial Contagion Networks as Annealing-Ready Ising Systems}\\[0.3em]
\large Cascades, Bailout Optimization, and Susceptibility}

\author{Abhinav Tomar}
\altaffiliation{These authors contributed equally to this work.}
\affiliation{The Institute of Mathematical Sciences (IMSc), C.I.T Campus, Taramani, Chennai 600113, India}
\affiliation{QCAR Group, The Institute of Mathematical Sciences (IMSc), Chennai 600113, India}
\author{Lakshya Nagpal}
\altaffiliation{These authors contributed equally to this work.}
\affiliation{The Institute of Mathematical Sciences (IMSc), C.I.T Campus, Taramani, Chennai 600113, India}
\affiliation{QCAR Group, The Institute of Mathematical Sciences (IMSc), Chennai 600113, India}
\affiliation{Pecslab Research}
\author{Vikas Chauhan}
\affiliation{Department of Physics, Ramjas College, University of Delhi, Delhi 110007, India}
\affiliation{QCAR Group, The Institute of Mathematical Sciences (IMSc), Chennai 600113, India}
\author{S.R. Hassan}
\affiliation{The Institute of Mathematical Sciences (IMSc), C.I.T Campus, Taramani, Chennai 600113, India}
\affiliation{QCAR Group, The Institute of Mathematical Sciences (IMSc), Chennai 600113, India}
\affiliation{Homi Bhabha National Institute, Anushakti Nagar, Mumbai, Maharashtra 400094}

\begin{abstract}
Interconnected financial systems are vulnerable to cascading failures arising from cross-holdings and nonlinear contagion, making the analysis and mitigation of systemic risk a challenging computational problem. In this work, we develop a unified optimization framework for financial network analysis based on Ising models and Quadratic Unconstrained Binary Optimization (QUBO). Starting from the Elliott--Golub--Jackson financial network model, we extend equilibrium valuation to incorporate threshold-induced failures, formulate the Maximum Cascade Failure Problem, and derive an equivalent QUBO representation. We then formulate the Optimal Bailout Allocation Problem as a controlled Ising model and transform the resulting bi-level optimization into a single joint QUBO that simultaneously determines equilibrium failures and optimal interventions under budget constraints. To characterize the influence of individual institutions, we introduce bailout susceptibility as a response-based measure of systemic importance and develop a susceptibility-driven greedy intervention strategy. Numerical simulations demonstrate equilibrium valuation, worst-case cascade identification, optimal bailout allocation, and susceptibility analysis on financial networks of varying sizes. 
Beyond optimization, the Ising representation provides a general statistical-mechanical framework for analyzing financial contagion, enabling the application of response theory, Monte Carlo methods, and other techniques developed for interacting spin systems.
The proposed framework establishes a unified approach for systemic risk analysis that is compatible with classical annealing, quantum-inspired optimization, and emerging quantum annealing technologies.
\end{abstract}

\maketitle

\section{Introduction}

Modern financial systems are highly interconnected through equity ownership, debt obligations, interbank lending, derivative contracts, and other forms of financial exposure.~\cite{eisenberg2001systemic,
gai2010contagion,
acemoglu2015systemic,
glasserman2016contagion,
haldane2011systemic,
battiston2012debtrank,
bardoscia2017pathways,
caccioli2014stability,
may2010systemic} While these interconnections improve liquidity and the allocation of capital, they also create channels through which financial distress can propagate across institutions. Consequently, the failure of a single institution may trigger a sequence of subsequent failures, amplifying a localized shock into a systemic crisis. The global financial crisis of 2008 demonstrated that systemic risk depends not only on the financial health of individual institutions but also on the structure of the network that connects them.
\cite{battiston2012debtrank,
haldane2011systemic,
poledna2015el,
battiston2016complexity,
bardoscia2017pathways} Understanding how shocks propagate through such networks and how systemic crises can be mitigated has therefore become a central problem in financial economics, network science, and risk management. \cite{newman2010networks,
barabasi2016network,
watts2002simple,
granovetter1978threshold,
acemoglu2012aggregate}

Several mathematical frameworks have been developed to study financial contagion and systemic risk. Early network models focused on interbank liabilities and clearing mechanisms, while subsequent approaches incorporated cross-holdings, default cascades, and measures of systemic importance. Representative examples include the clearing model of Eisenberg and Noe, the financial network model of Elliott, Golub, and Jackson (EGJ), DebtRank, and related network-based stress-testing methodologies~\cite{eisenberg2001systemic,
rogers2013failure,
elliott2014financial,
gai2010contagion,
acemoglu2015systemic,
glasserman2016contagion,
battiston2012debtrank,
battiston2016complexity,
bardoscia2017pathways,
caccioli2014stability,
poledna2015el}. Among these, the EGJ model provides a particularly convenient starting point because it describes financial institutions as interconnected through cross-holdings while simultaneously owning primitive assets whose values are determined outside the network. The market values of all institutions are obtained from a system of coupled linear equations, naturally capturing the amplification of shocks through direct and indirect ownership chains.

Although equilibrium valuation can be computed efficiently, the introduction of financial distress fundamentally changes the problem. Institutions become vulnerable when their market values fall below critical solvency thresholds, triggering additional losses arising from liquidation costs, fire-sale effects, funding shortages, or loss of market confidence. These threshold-induced losses introduce nonlinear feedback into the valuation equations, transforming the original linear equilibrium problem into a nonlinear fixed-point system capable of generating cascading failures. Consequently, relatively small asset shocks may produce disproportionately large systemic responses.

The emergence of nonlinear contagion naturally leads to optimization problems of direct relevance to financial regulation. Rather than asking only whether a financial network is stable, regulators seek to identify the largest cascade that can arise from an initial disturbance and the most effective allocation of limited bailout resources to suppress such cascades. These problems involve optimization over an exponentially large space of possible failure and intervention configurations, making exact solution computationally impractical for large financial systems.

Quadratic Unconstrained Binary Optimization (QUBO) provides a natural framework for addressing such combinatorial problems. QUBO formulations are equivalent to Ising Hamiltonians and therefore establish a direct connection between optimization, statistical mechanics, and annealing-based computation. The same mathematical formulation can be solved using classical simulated annealing, simulated quantum annealing, quantum-inspired optimizers, or quantum annealing hardware, making QUBO an attractive framework for large-scale optimization problems~\cite{kirkpatrick1983optimization,kadowaki1998quantum,johnson2011quantum}

In this work, we develop a unified optimization framework for systemic risk analysis that combines equilibrium valuation, contagion dynamics, combinatorial optimization, and statistical-physics methods within a common mathematical formulation. We first extend the EGJ equilibrium valuation model by incorporating threshold-induced failures, thereby transforming the linear valuation equations into a nonlinear fixed-point problem describing cascading contagion. Building upon this formulation, we introduce the \emph{Maximum Cascade Failure Problem} (MCFP), formulate it as a QUBO optimization problem, and derive an equivalent Ising representation whose minimum-energy configurations represent feasible contagion cascades.

We then consider the complementary problem of regulatory intervention by formulating the \emph{Optimal Bailout Allocation Problem} (OBAP). Bailout actions are introduced as external control fields acting on an Ising model of financial contagion, allowing the original bi-level optimization problem to be reformulated as a single joint QUBO that simultaneously determines equilibrium failures and optimal interventions under budget constraints.

Finally, motivated by linear-response theory, we introduce the concept of \emph{bailout susceptibility} as a quantitative measure of systemic importance. The resulting susceptibility matrix characterizes how interventions applied to one institution influence the stability of others and naturally leads to a response-based greedy bailout strategy that provides an efficient approximation to the full joint optimization problem.

The present work lies at the intersection of financial network theory, combinatorial optimization, and statistical physics. By establishing an explicit mapping between financial contagion and Ising optimization, it provides a unified computational framework for equilibrium valuation, worst-case cascade analysis, and optimal regulatory intervention. Although the numerical examples employ annealing-based optimization methods, the mathematical formulations are independent of any specific computational platform and are equally applicable to classical optimizers, quantum-inspired algorithms, quantum annealers, and future gate-based ~\cite{farhi2014qaoa} quantum optimization methods.

While previous work has established the dynamics of threshold financial networks and their computational complexity, no existing formulation places these systems within the framework of equilibrium statistical mechanics. By constructing an equivalent Ising representation, we obtain not only an optimization model but also access to a broad range of analytical tools developed for interacting spin systems. These include annealing algorithms, response theory, correlation analysis, and other techniques that become directly applicable to financial contagion.

The remainder of this paper is organized as follows. Section~II develops the financial network valuation model together with threshold-induced contagion dynamics. Section~III formulates the Maximum Cascade Failure Problem, establishes its computational complexity, and derives its QUBO representation. Section~IV introduces the controlled Ising formulation of optimal bailout allocation and the corresponding joint QUBO model. Section~V develops the bailout susceptibility framework and presents a response-based greedy intervention strategy. Section~VI presents numerical results, and Section~VII concludes the paper.

\section{Financial Network Valuation and Threshold Contagion}

We consider an interconnected financial system consisting of $n$ financial institutions and $m$ primitive assets. The framework follows the financial network model of Elliott, Golub, and Jackson~\cite{elliott2014financial}, which combines direct asset ownership with cross-holdings between institutions. Our objective is to extend this equilibrium valuation model by incorporating threshold-induced failures that generate nonlinear contagion dynamics.

Let
\[
\mathcal I=\{1,2,\ldots,n\}
\]
denote the set of financial institutions and
\[
\mathcal A=\{1,2,\ldots,m\}
\]
the set of primitive assets. The market prices of the primitive assets are collected into the vector
\begin{equation}
p=(p_1,p_2,\ldots,p_m)^{\top}\in\mathbb{R}_{\ge0}^{m}.
\end{equation}

The ownership of primitive assets is described by the non-negative matrix
\begin{equation}
D\in\mathbb{R}_{\ge0}^{n\times m},
\end{equation}
where $D_{ik}$ denotes the fraction of asset $k$ owned by institution $i$. Since the total ownership of each asset cannot exceed its full value,
\begin{equation}
\sum_{i=1}^{n}D_{ik}\le1,
\qquad
k=1,\ldots,m.
\end{equation}

Financial institutions are further connected through equity cross-holdings represented by the matrix
\begin{equation}
C\in\mathbb{R}_{\ge0}^{n\times n},
\end{equation}
where $C_{ij}$ denotes the fraction of institution $j$ owned by institution $i$. Throughout this work we assume
\begin{equation}
C_{ii}=0,
\qquad
\sum_{i=1}^{n}C_{ij}\le1,
\end{equation}
excluding self-ownership and ensuring that no institution is owned beyond its total equity.

The remaining ownership belongs to investors outside the financial network. Defining the outside ownership fractions by
\begin{equation}
\hat C_{jj}
=
1-
\sum_{i=1}^{n}C_{ij},
\end{equation}
the corresponding diagonal matrix is
\begin{equation}
\hat C
=
\mathrm{diag}
(\hat C_{11},\ldots,\hat C_{nn}).
\end{equation}
We assume
\begin{equation}
\hat C_{jj}>0,
\qquad
j=1,\ldots,n,
\end{equation}
so that every institution possesses a non-zero fraction of outside ownership.

The equilibrium equity value of institution $i$ consists of the value of the primitive assets that it directly owns together with the value inherited through its ownership of other financial institutions. Writing
\[
V=(V_1,\ldots,V_n)^{\top},
\]
the equilibrium condition becomes
\begin{equation}
V=Dp+CV.
\label{eq:equity_matrix_new}
\end{equation}
Since $\rho(C)<1$, the matrix $(I-C)$ is invertible and Eq.~(\ref{eq:equity_matrix_new}) has the unique solution
\begin{equation}
V=(I-C)^{-1}Dp.
\label{eq:equity_solution_new}
\end{equation}

The quantity relevant for systemic risk is the market value visible to investors outside the network,
\begin{equation}
v=\hat C V,
\end{equation}
which may be written as
\begin{equation}
v
=
\hat C(I-C)^{-1}Dp.
\end{equation}
Introducing the dependency matrix
\begin{equation}
A=\hat C(I-C)^{-1},
\label{eq:dependency_matrix}
\end{equation}
the equilibrium valuation assumes the compact form
\begin{equation}
v=ADp.
\label{eq:linear_equilibrium}
\end{equation}

The matrix $A$ plays a central role throughout this work~\cite{elliott2014financial,
acemoglu2012aggregate,
glasserman2016contagion}. It incorporates not only direct ownership relations but also all indirect ownership chains generated through repeated cross-holdings. Consequently, every element $A_{ij}$ measures the effective influence of institution $j$ on the market value of institution $i$.

Equation~(\ref{eq:linear_equilibrium}) describes a stable linear valuation model in which equilibrium market values are obtained through a single matrix inversion. Financial contagion emerges only after introducing threshold-induced failures.

To model financial distress, we associate with each institution a critical market value
\begin{equation}
v_i^{c}>0,
\end{equation}
below which the institution is regarded as insolvent,
\begin{equation}
v_i<v_i^{c}.
\label{eq:failure_condition_new}
\end{equation}
Failure generates an additional loss
$\beta_i>0$,
representing liquidation costs, funding shortages, fire-sale effects, or loss of market confidence. \cite{orus2019forecasting}
Using the Heaviside step function,
\begin{equation}
\Theta(x)=
\begin{cases}
1,&x\ge0,\\
0,&x<0,
\end{cases}
\end{equation}
the failure penalty is written as
\begin{equation}
b_i(v_i)
=
\beta_i
\left[
1-\Theta(v_i-v_i^{c})
\right].
\end{equation}
Collecting all penalties into the vector
\[
b(v)=\bigl(b_1,\ldots,b_n\bigr)^{\top},
\]
the equilibrium valuation becomes
\begin{equation}
v
=
A\left(Dp-b(v)\right).
\label{eq:nonlinear_equilibrium_new}
\end{equation}

Equation~(\ref{eq:nonlinear_equilibrium_new}) is a nonlinear fixed-point equation because the market values determine the failure penalties while the penalties simultaneously modify the market values. The introduction of threshold effects therefore transforms the linear valuation model into a nonlinear contagion problem capable of generating multiple self-consistent equilibria.

Starting from the linear equilibrium
\[
v^{(0)}=ADp,
\]
an initial shock triggers a sequence of failure sets
\[
S^{(0)}
\subseteq
S^{(1)}
\subseteq
S^{(2)}
\subseteq
\cdots,
\]
where the market values evolve according to
\begin{equation}
v_i^{(t)}
=
v_i^{(0)}
-
\sum_{j\in S^{(t-1)}}A_{ij}\beta_j,
\label{eq:cascade_iteration_new}
\end{equation}
and newly failed institutions satisfy
\begin{equation}
S^{(t)}
=
\left\{
i\,\big|\,
v_i^{(t)}<v_i^{c}
\right\}.
\label{eq:cascade_set_new}
\end{equation}

Since the number of institutions is finite, the iteration converges after a finite number of steps to a self-consistent failure configuration. This nonlinear fixed-point structure provides the foundation for the optimization problems studied in the following sections, namely the determination of the largest possible cascade generated by an initial shock and the optimal allocation of limited bailout resources to suppress systemic contagion.

\section{Maximum Cascade Failure and QUBO Encoding}
\label{sec:mcfp-ising}

The nonlinear fixed-point formulation developed in the previous section determines the equilibrium failure configuration of a financial network once the failure mechanism has been specified. An important question in systemic risk analysis is to identify the most fragile configuration that is compatible with the network structure. This leads naturally to the Maximum Cascade Failure Problem (MCFP), which seeks the largest self-consistent failure configuration permitted by the financial dependency graph.

It is important to emphasize that the dependency matrix $A$ is \emph{not} an arbitrary adjacency matrix. Rather, $A=\hat C(I-C)^{-1}$ is obtained from the equilibrium valuation model of Elliott, Golub, and Jackson \cite{elliott2014financial} and therefore incorporates both direct and indirect cross-holding dependencies through all ownership paths. Consequently, the threshold graph induced by $A_{ij}>\theta$ represents equilibrium contagion dependencies derived from the financial network rather than an independently specified graph.

To describe the failure state of the network, we associate with each institution a binary variable
\begin{equation}
x_i=
\begin{cases}
1,&\text{institution $i$ fails},\\
0,&\text{institution $i$ survives},
\end{cases}
\qquad
i=1,\ldots,n.
\end{equation}
The total cascade size is therefore
\begin{equation}
F(x)=\sum_{i=1}^{n}x_i.
\end{equation}

The objective of the Maximum Cascade Failure Problem is to determine the binary configuration
\begin{equation}
x^\star
=
\arg\max_{x\in\{0,1\}^n}
F(x),
\label{eq:MCFP}
\end{equation}
subject to the logical constraints imposed by the financial dependency network.

The dependency matrix $A$ determines how failures propagate through the financial network. Whenever

\begin{equation}
A_{ij}>\theta,
\end{equation}

the failure of institution $j$ necessarily induces the failure of institution $i$. This contagion rule is represented by the logical implication

\begin{equation}
x_j=1
\Longrightarrow
x_i=1,
\label{eq:implication}
\end{equation}

where $x_i\in\{0,1\}$ denotes the failure state of institution $i$. Consequently, admissible failure configurations are precisely those that are closed under the contagion dynamics.

The resulting optimization problem seeks the self-consistent failure configuration containing the maximum number of failed institutions,

\begin{equation}
\max_{\mathbf{x}\in\{0,1\}^N}
\sum_{i=1}^{N}x_i,
\end{equation}

subject to the logical contagion constraints induced by Eq.~(\ref{eq:implication}). We refer to this optimization problem as the Maximum Cascade Failure Problem (MCFP).

The computational complexity of worst-case failure cascades in the Elliott--Golub--Jackson financial network model has already been established by Hemenway and Khanna-\cite{hemenway2016sensitivity,schuldenzucker2017,garey1979computers}. They showed, via a reduction from the Balanced Complete Bipartite Subgraph (BCBS) problem, that determining the maximum number of failures resulting from bounded asset-price shocks is NP-hard. Their result demonstrates that exhaustive search for worst-case cascades is computationally intractable in general, thereby motivating optimization-based approaches.

In the following section, we reformulate the MCFP as a Quadratic Unconstrained Binary Optimization (QUBO) problem, enabling the use of annealing-based optimization methods.

To favour large cascades, we introduce the reward functional
\begin{equation}
E_{\rm reward}
=
-w
\sum_{i=1}^{n}x_i,
\qquad
w \geq 0,
\end{equation}
whose minimization encourages configurations containing many failed institutions.

The contagion constraints are incorporated through quadratic penalty functions. The implication
\[
x_j\Rightarrow x_i
\]
is violated only by the local configuration
\[
(x_i,x_j)=(0,1),
\]
which is detected by
\begin{equation}
P_{ij}
=
x_j(1-x_i).
\end{equation}

Summing over all active dependency links gives the constraint energy
\begin{equation}
E_{\rm con}
=
\lambda
\sum_{A_{ij}>\theta}
A_{ij}
x_j(1-x_i),
\end{equation}
where $\lambda>0$ determines the penalty assigned to violations of the contagion dynamics.

To specify the initiating disturbance, the seed institution $s$ is fixed through the penalty
\begin{equation}
E_{\rm seed}
=
M(1-x_s),
\qquad
M\gg\lambda>w,
\end{equation}
which guarantees that the chosen seed remains in the failed state throughout the optimization.

Combining the reward, contagion, and seed contributions yields the complete QUBO objective,
\begin{equation}
E(x)
=
-w\sum_i x_i
+
\lambda
\sum_{A_{ij}>\theta}
A_{ij}
x_j(1-x_i)
+
M(1-x_s),
\label{eq:MCFP_QUBO}
\end{equation}
and the optimal cascade is obtained from
\begin{equation}
x^\star
=
\arg\min_{x\in\{0,1\}^n}
E(x).
\end{equation}

Expanding Eq.~(\ref{eq:MCFP_QUBO}) gives the standard quadratic form
\begin{equation}
E(x)=x^{\mathrm T}Qx+\mathrm{const},
\end{equation}
where the matrix $Q$ is determined by the reward coefficients, the dependency matrix, and the penalty parameters.

Introducing Ising spin variables \cite{lucas2014ising}
\begin{equation}
x_i=\frac{1+s_i}{2},
\qquad
s_i\in\{-1,+1\},
\end{equation}
transforms the QUBO directly into an Ising Hamiltonian. Within this representation, the reward term acts as an effective external field favouring failure, while the quadratic couplings enforce the logical contagion constraints encoded by the financial dependency graph. Consequently, determining the largest self-consistent cascade becomes equivalent to finding the ground state of an interacting Ising system, allowing the optimization to be performed using simulated annealing, simulated quantum annealing, parallel-tempered annealing, or quantum annealing hardware.

\section{Optimal Bailout Allocation as Controlled Ising Optimization}
\label{sec:obap-ising}

The Maximum Cascade Failure Problem identifies the most severe self-consistent failure configuration permitted by the financial network. From the perspective of financial regulation, however, the objective is fundamentally different. Rather than maximizing systemic collapse, the regulator seeks the smallest intervention capable of preventing large-scale contagion~\cite{acharya2007many} while operating under limited financial resources.~\cite{capponi2022optimal,jackson2024credit} This leads naturally to the Optimal Bailout Allocation Problem (OBAP). \cite{dasaratha2024diversified}

To formulate the problem, we associate with each institution a binary bailout variable
\begin{equation}
b_i=
\begin{cases}
1,&\text{institution $i$ is rescued},\\
0,&\text{otherwise},
\end{cases}
\qquad
i=1,\ldots,n.
\end{equation}
Each bailout incurs a financial cost
\begin{equation}
\kappa_i>0,
\end{equation}
so that the total intervention expenditure is
\begin{equation}
C(b)
=
\sum_{i=1}^{n}\kappa_i b_i.
\label{eq:bailout_cost}
\end{equation}

To describe the collective dynamics of failures and recoveries, we employ the Ising representation introduced in the previous section. The state of institution $i$ is represented by the spin variable
\begin{equation}
s_i=
\begin{cases}
+1,&\text{failed},\\
-1,&\text{surviving}.
\end{cases}
\label{eq:spin_definition}
\end{equation}

In the absence of regulatory intervention the financial network is described by the effective Ising Hamiltonian
\begin{equation}
H_0(s)
=
-\sum_{i<j}J_{ij}s_is_j
-
\sum_{i=1}^{n}h_i s_i,
\label{eq:bare_hamiltonian}
\end{equation}
where the couplings $J_{ij}$ describe the propagation of contagion through the dependency network and the local fields $h_i$ quantify the intrinsic financial vulnerability of individual institutions.

A bailout acts as an external stabilizing field that favours the surviving state. The controlled Hamiltonian therefore becomes
\begin{equation}
H_{\mathrm{ctrl}}(s,b)
=
-\sum_{i<j}J_{ij}s_is_j
-
\sum_{i=1}^{n}
\left(
h_i-\Omega_i b_i
\right)s_i,
\label{eq:controlled_hamiltonian}
\end{equation}
where
\begin{equation}
\Omega_i>0
\end{equation}
denotes the stabilizing strength of the bailout applied to institution $i$.

Equation~(\ref{eq:controlled_hamiltonian}) has a transparent physical interpretation. The interaction term promotes correlated failures through financial contagion, whereas the bailout field locally biases selected institutions toward the surviving state. Systemic stabilization therefore emerges through a competition between collective contagion and targeted regulatory intervention.

The following proposition establishes a sufficient condition under which a bailout guarantees the survival of a rescued institution.

\begin{proposition}[Bailout pinning]
If
\begin{equation}
\Omega_i
>
|h_i|
+
\sum_{j\neq i}|J_{ij}|,
\label{eq:pinning_condition}
\end{equation}
then institution $i$ is pinned to the surviving state
\begin{equation}
s_i=-1
\end{equation}
in the ground state of the controlled Hamiltonian whenever
\begin{equation}
b_i=1.
\end{equation}
\end{proposition}

\begin{proof}
Consider two configurations differing only in the value of spin $s_i$, while all remaining spins are held fixed. The largest possible interaction energy favouring failure is bounded by
\[
|h_i|+\sum_{j\neq i}|J_{ij}|.
\]
If the bailout field exceeds this bound, the energy of the surviving configuration remains lower than that of the failed configuration irrespective of the neighbouring spins. Consequently the ground state necessarily satisfies $s_i=-1$.
\end{proof}

The regulator must determine the subset of institutions that should receive financial assistance while respecting the available bailout budget. This naturally leads to a joint optimization over both the failure variables and the bailout variables.

Let
\begin{equation}
x=(x_1,\ldots,x_n)
\end{equation}
denote the binary failure variables and
\begin{equation}
b=(b_1,\ldots,b_n)
\end{equation}
the bailout variables. Suppose that $K$ institutions may be rescued,
\begin{equation}
\sum_{i=1}^{n}b_i = K.
\label{eq:budget_constraint}
\end{equation}

Using a quadratic penalty to enforce the budget constraint, the Optimal Bailout Allocation Problem is formulated as the joint QUBO
\begin{equation}
L(x,b)
=
\kappa^{\mathrm T}b
+
\lambda_1 x^{\mathrm T}Q(b)x
+
\lambda_2
\left(
\sum_i b_i-K
\right)^2,
\label{eq:joint_objective}
\end{equation}
where the three terms respectively represent

\begin{enumerate}
\item the total bailout expenditure,
\item the equilibrium cascade energy under the chosen intervention,
\item the quadratic penalty enforcing the bailout budget.
\end{enumerate}

The optimal intervention strategy is obtained by solving
\begin{equation}
(x^\ast,b^\ast)
=
\arg\min_{x,b}
L(x,b),
\label{eq:joint_problem}
\end{equation}
with
\[
x,b\in\{0,1\}^n.
\]

Unlike the Maximum Cascade Failure Problem, which searches for the most destructive equilibrium configuration, the joint QUBO simultaneously determines the least costly intervention strategy together with the resulting equilibrium failure state. The optimization therefore balances two competing objectives: minimizing public expenditure while maximizing systemic stability.

The formulation provides a unified optimization framework for regulatory intervention and can be solved using the same Ising-based algorithms employed for the cascade optimization problem, including simulated annealing, simulated quantum annealing, parallel-tempered annealing, and quantum annealing hardware. As demonstrated in the numerical results of Section~\ref{sec:results_discussion}, the resulting optimal bailout policies suppress contagion by strategically stabilizing a relatively small number of institutions, thereby interrupting the propagation of systemic failures throughout the financial network.

\section{Bailout Susceptibility and Greedy Intervention}
\label{sec:susceptibility}

One of the principal advantages of the Ising formulation is that it immediately enables the application of response theory. In statistical mechanics, linear response is characterized by the susceptibility tensor, which quantifies the change in equilibrium spin expectation values under infinitesimal external perturbations.~\cite{kubo1957statistical} Under the present mapping, bailout incentives act as local external fields, naturally leading to an analogous financial susceptibility.

The joint QUBO formulation developed in the previous section yields an optimal bailout strategy by solving a global combinatorial optimization problem. Although this approach provides the minimum-cost intervention, solving the full QUBO may become computationally demanding for very large financial networks. It is therefore desirable to identify local quantities that quantify the influence of individual institutions on the stability of the entire system and can guide efficient approximate intervention strategies.

In statistical physics, the response of a system to an external perturbation is characterized by its susceptibility. Motivated by this idea, we introduce a bailout susceptibility that measures how strongly the stability of one institution is affected by a small stabilizing intervention applied to another.

Within the Ising representation, let
\begin{equation}
\sigma_i^{z}
=
\begin{cases}
+1,&\text{failed},\\
-1,&\text{surviving},
\end{cases}
\end{equation}
denote the spin operator associated with institution $i$. The expectation value
\begin{equation}
\langle\sigma_i^z\rangle
\end{equation}
therefore measures the average failure tendency of institution $i$.
\begin{definition}[Bailout susceptibility]

For an equilibrium Ising system, the fluctuation--dissipation theorem relates the linear response of the system to its equilibrium fluctuations. We therefore define the bailout susceptibility by

\begin{equation}
\chi_{ij}
=
\beta_{\rm eq}
\left(
\langle
\sigma_i^z\sigma_j^z
\rangle
-
\langle
\sigma_i^z
\rangle
\langle
\sigma_j^z
\rangle
\right),
\label{eq:bailout_susceptibility}
\end{equation}

where $\beta_{\rm eq}=1/(k_BT_{\rm eq})$ denotes the inverse equilibrium temperature.

\end{definition}

Equation~(\ref{eq:bailout_susceptibility}) provides the operational definition used throughout this work. All susceptibilities reported in Section~\ref{sec:results_discussion} are computed directly from equilibrium Metropolis Monte Carlo samples using the fluctuation--dissipation theorem.

Positive values of $\chi_{ij}$ indicate that stabilizing institution $j$ reduces the equilibrium failure tendency of institution $i$, whereas values close to zero imply weak dynamical coupling between the two institutions.

To quantify the overall influence of a rescued institution, we define its aggregate susceptibility
\begin{equation}
\mu_j
=
\frac1n
\sum_{i=1}^{n}
\chi_{ij}.
\label{eq:aggregate_susceptibility}
\end{equation}

The quantity $\mu_j$ measures the average reduction in systemic failure produced by rescuing institution $j$ and therefore provides a dynamical measure of systemic importance. Unlike purely topological centrality measures, $\mu_j$ incorporates both the network connectivity and the collective contagion dynamics encoded by the Ising model.

To characterize the overall health of the financial system, we introduce the average magnetization
\begin{equation}
m(b)
=
\frac1n
\sum_{i=1}^{n}
\langle\sigma_i^z\rangle_b,
\label{eq:magnetization}
\end{equation}
which satisfies
\[
-1\le m(b)\le1.
\]

Positive magnetization corresponds to a failure-dominated phase, whereas negative magnetization indicates that most institutions remain solvent. The objective of regulatory intervention is therefore to minimize the network magnetization.

The sensitivity of the global financial system to a bailout applied to institution $j$ is obtained directly from Eq.~(\ref{eq:magnetization}),
\begin{equation}
-
\frac{\partial m}{\partial\Omega_j}
=
\frac1n
\sum_i\chi_{ij}
=
\mu_j,
\end{equation}
showing that aggregate susceptibility measures the marginal reduction in systemic distress produced by rescuing institution $j$.

When bailout costs differ among institutions, it is natural to compare systemic benefit with intervention cost. We therefore define the efficiency ratio
\begin{equation}
\rho_j
=
\frac{\mu_j}{\kappa_j},
\label{eq:efficiency_ratio}
\end{equation}
which measures the expected stabilization achieved per unit bailout cost.

The efficiency ratio leads naturally to a scalable approximation algorithm.

\begin{definition}[Susceptibility-driven greedy intervention]
Suppose that the regulator may rescue at most $K$ institutions. Starting from an empty bailout set,
\[
S=\varnothing,
\]
the intervention strategy proceeds iteratively by selecting
\begin{equation}
j^\ast
=
\arg\max_{j\notin S}
\frac{\mu_j}{\kappa_j},
\end{equation}
adding $j^\ast$ to the bailout set, recomputing the susceptibilities, and repeating until $|S|=K$.
\end{definition}\cite{nemhauser1978analysis,
kempe2003maximizing,
kitsak2010identification}

The greedy algorithm requires only local response information and avoids solving the full joint QUBO at every iteration, making it attractive for large-scale financial networks.

The quality of the greedy solution can be characterized under standard assumptions from combinatorial optimization. Let
\begin{equation}
f(S)
=
-m(S)
\end{equation}
denote the reduction in systemic failure achieved by rescuing the institution set $S$. If $f(S)$ is monotone and submodular, then the classical theorem of Nemhauser, Wolsey, and Fisher \cite{nemhauser1978analysis} implies
\begin{equation}
f(S_{\rm greedy})
\ge
\left(1-\frac1e\right)
f(S_{\rm opt}),
\label{eq:approximation_bound}
\end{equation}
where $S_{\rm opt}$ denotes the optimal bailout set satisfying the same budget constraint.

The approximation guarantee is conditional upon the monotonicity and submodularity of the cascade-reduction function and therefore need not hold for arbitrary financial networks. Nevertheless, many contagion processes exhibit diminishing marginal returns, making the susceptibility-driven greedy strategy a practical approximation when solving the full joint QUBO is computationally prohibitive.

The susceptibility framework therefore complements the exact optimization developed in the previous section. The joint QUBO yields globally optimal bailout policies, while the susceptibility provides a physically motivated measure of systemic importance and forms the basis of scalable response-based intervention strategies. As demonstrated in Section~\ref{sec:results_discussion}, the susceptibility matrix offers direct insight into the pathways through which stabilization propagates across the financial network.

\section{Results and Discussion}
\label{sec:results_discussion}

In this section, we present numerical results that validate the theoretical framework developed in the preceding sections. The simulations follow the same progression as the mathematical formulation: we first examine equilibrium valuation and the emergence of threshold-induced failures, then investigate the Maximum Cascade Failure Problem through its QUBO formulation, followed by the optimization of bailout strategies using the controlled Ising model, and finally analyse the bailout susceptibility that quantifies the systemic importance of individual institutions.

The results are organized according to these four components. For each problem, we discuss the corresponding numerical solution together with its physical interpretation and relate it to the appropriate figure. Collectively, the simulations demonstrate how equilibrium valuation, financial contagion, worst-case cascade analysis, and optimal regulatory intervention can all be formulated within a unified Ising--QUBO framework and solved using annealing-based optimization methods.

\subsection{Computational Framework}

All numerical optimization experiments presented in this work were performed
using \textsc{QAnneal}~\cite{qanneal2026}, an in-house computational framework developed by the
authors for solving Ising Hamiltonians and Quadratic Unconstrained Binary
Optimization (QUBO) problems. The framework provides a unified implementation~\cite{ding2019towards}
of several annealing paradigms, including Simulated Annealing (SA), \cite{lucas2014ising,
glover2019,
johnson2011quantum}
Simulated Quantum Annealing (SQA), Simulated Quantum Annealing with Parallel
Tempering (SQAPT), and Continuous-Time Path-Integral Monte Carlo (CT-PIMC), \cite{santoro2002theory}
allowing the same optimization problem to be investigated using both classical
and quantum-inspired annealing dynamics.\cite{kadowaki1998quantum}

In the present work, \textsc{QAnneal} serves exclusively as the numerical
backend for solving the QUBO formulations derived from the financial contagion
model. The software itself is independent of the financial application
considered here and is designed as a general-purpose framework for Ising and
QUBO optimization. A detailed description of its architecture, implementation,
algorithmic design, and benchmarking against existing optimization frameworks
will be presented in a separate software publication.

Unless otherwise stated, all optimization results reported in this paper were
obtained using the \textsc{QAnneal} framework.

An alpha release of implementation is available at PyPi \cite{qanneal2026}. 

\subsection{Equilibrium market valuation}

We begin by validating the equilibrium valuation framework developed in Secs.~II and III. The nonlinear fixed-point equation,
\begin{equation}
v=A\left(Dp-b(v)\right),
\end{equation}
was solved for a synthetic financial network consisting of $n=20$ interconnected institutions holding $m=5$ primitive assets. The market values were represented using a binary expansion of order $q=2$, while the discontinuous Heaviside failure function was approximated by a third-order Legendre polynomial to obtain a QUBO-compatible formulation.

Figure~\ref{fig:market_equ} compares, for each institution, the initial market value $v_i^{(0)}$, the critical solvency threshold $v_i^c$, and the equilibrium market value $v_i^{*}$ obtained after solving the nonlinear valuation problem. The institutions are grouped into three tiers mega-banks, regional institutions, and peripheral institutions illustrating the heterogeneous structure of the financial network.

The equilibrium solution shows that all but one institution remain above their solvency thresholds. Institution~8 crosses the critical threshold,
\begin{equation}
v_8^{*}<v_8^{c},
\end{equation}
and is therefore identified as failed. The resulting equilibrium loss is approximately $9.69$, corresponding to about $12.5\%$ of the total initial market value.

This example illustrates the qualitative effect of introducing threshold-induced failures into the valuation model. In the absence of the nonlinear penalty term, the equilibrium is obtained directly from the linear relation $v=ADp$. Once failure thresholds are incorporated, however, the valuation problem becomes nonlinear and admits equilibrium failure configurations in which relatively modest changes in market values can trigger institution defaults. The equilibrium failure set obtained here provides the initial condition for the cascade optimization problems considered in the following sections.

\subsection{Maximum cascade failure}

We next investigate the Maximum Cascade Failure Problem (MCFP) using the QUBO formulation of Eq.~(\ref{eq:MCFP_QUBO}). The optimization was performed on a synthetic financial network containing $n=1000$ interconnected institutions using simulated quantum annealing (SQA). To verify the robustness of the solution, the optimization was independently repeated using simulated quantum annealing with parallel tempering (SQAPT).

Figure~\ref{fig:maxcasc} shows the worst-case cascade identified by the optimizer. The orange node denotes the initially shocked seed institution, while red and green nodes represent institutions that eventually fail or survive, respectively. Network edges indicate dependency relationships whose strengths exceed the threshold $\theta$, so that failure propagates only along sufficiently strong financial exposures.

The optimized configuration contains a self-consistent cascade involving $561$ of the $1000$ institutions, corresponding to more than half of the financial network, with a minimum QUBO energy of $-498.50$. The cascade propagates across all hierarchical levels of the network: all $10$ mega-banks and $89$ regional institutions fail, while $462$ of the $900$ peripheral institutions are also driven into insolvency. The remaining $438$ peripheral institutions survive because they are either weakly connected to the principal contagion cluster or lie outside the dominant failure pathways.

The independent SQAPT simulation converges to the same cascade size and nearly identical minimum energy, providing additional evidence that the reported solution is robust and not an artifact of a particular annealing schedule. These results demonstrate that the proposed QUBO formulation successfully identifies large self-consistent failure configurations without exhaustive enumeration over the exponentially large configuration space, illustrating the applicability of annealing-based optimization to worst-case systemic risk analysis.

\subsection{Optimal bailout allocation}

We now apply the joint QUBO formulation of the Optimal Bailout Allocation Problem introduced in Eq.~(\ref{eq:joint_objective}). Unlike the Maximum Cascade Failure Problem, which identifies the largest self-consistent failure configuration, the present optimization determines the minimum-cost intervention capable of suppressing systemic contagion under a prescribed bailout budget. The optimization simultaneously identifies both the institutions to be rescued and the resulting equilibrium state of the financial network.

The optimization mechanism is validated across financial networks ranging
from $20$ to $1000$ institutions. In each case, the joint QUBO objective is
minimized to simultaneously determine the equilibrium failure
configuration and the optimal bailout allocation under a fixed budget
constraint, rather than selecting institutions by size or degree alone.

For the $20$-institution network, the bailout drives the magnetization
from $m=+0.400$ to $m=-0.800$. For the $100$-institution hierarchical
network (core, intermediate, peripheral layers), it changes from
$m=+0.840$ to $m=-0.060$. For the $1000$-institution network, it changes
from $m=+0.092$ to $m=-0.482$, a net reduction $\Delta m=-0.574$.

The $n=100$ case is shown, in Figure~\ref{fig:bailout_100}: the left
panel shows the pre-intervention cascade, the center panel the optimal
bailout allocation, and the right panel the post-intervention equilibrium,
in which many institutions survive without direct assistance because
contagion paths to them are blocked. Across all three sizes, the same
mechanism builds efficient firebreaks that stabilize a much larger
fraction of the network than the number of institutions directly rescued. The corresponding state transitions are shown in Figure~\ref{fig:sankey}.
 
Across all three system sizes, the same joint QUBO consistently constructs efficient ``firebreaks'' that interrupt contagion pathways, allowing relatively few directly rescued institutions to stabilize a much larger fraction of the financial network. These results demonstrate that the proposed optimization framework captures not only the mechanism of optimal intervention but also its robustness and scalability across increasingly complex financial systems.

\subsection{Annealing dynamics and optimization landscape}

The QUBO formulations developed in this work define rugged optimization landscapes containing numerous local minima separated by energy barriers. Such landscapes are typical of large combinatorial optimization problems and present a significant challenge for deterministic optimization algorithms, which can easily become trapped in metastable configurations. Simulated annealing provides a stochastic optimization strategy that mitigates this difficulty by combining thermal exploration with gradual cooling. \cite{kirkpatrick1983optimization}

Figure~\ref{fig:simanneal} illustrates this optimization process. The left panel shows a representative energy landscape, where darker regions correspond to lower objective values and the yellow marker indicates the global minimum. The multiple valleys visible in the landscape highlight the existence of competing local minima that characterize Ising-QUBO optimization problems.

The right panel shows a typical annealing trajectory generated using the Metropolis algorithm. The optimization begins from a random initial configuration (cyan). During the early stages of the annealing schedule, when the temperature is high, the algorithm accepts not only energy-lowering moves but also a finite fraction of energy-increasing moves with probability \cite{metropolis1953equation}

\begin{equation}
P_{\rm acc}
=
\min\!\left(1,e^{-\Delta E/T}\right).
\end{equation}

These thermal fluctuations enable the trajectory to cross energy barriers and explore different regions of the configuration space rather than becoming trapped in the first local minimum encountered.

As the temperature is gradually reduced, the acceptance probability for uphill moves decreases exponentially. Consequently, most proposed moves are rejected (red), while only energetically favourable moves continue to be accepted (green). The optimization therefore evolves from an exploratory search at high temperature to a highly selective local search at low temperature, ultimately converging toward a low-energy basin close to the global optimum.

The same annealing strategy is employed throughout this work to solve the QUBO formulations associated with equilibrium valuation, maximum cascade failure, and optimal bailout allocation. Although the financial optimization problems involve substantially larger and more complex energy landscapes than the illustrative example shown here, the underlying optimization principle remains identical: controlled thermal fluctuations permit efficient exploration of the combinatorial search space before the system progressively freezes into a near-optimal solution.

\subsection{Bailout susceptibility and emergent systemic importance}

The previous results demonstrated that the joint QUBO formulation identifies targeted bailout policies capable of suppressing large financial cascades. An equally important question, however, is \emph{why} certain institutions are consistently selected by the optimization. The answer is provided by the bailout susceptibility introduced in Sec.~\ref{sec:susceptibility}, which measures the response of the equilibrium state to an infinitesimal stabilizing intervention.

Figure~\ref{fig:susceptibility_matrices} illustrates the complete progression from the original financial network to its physical response properties. The panels represent successive levels of description:
\[
A
\;\longrightarrow\;
J
\;\longrightarrow\;
\chi,
\]
namely the financial dependency matrix, and the equilibrium susceptibility matrix. This sequence provides the conceptual bridge connecting financial contagion with statistical mechanics.

The left panel shows the dependency matrix $A=\hat C(I-C)^{-1}$ introduced in Eq.~(\ref{eq:dependency_matrix}). The block structure reflects the hierarchical organization of the synthetic financial system into core, intermediate, and peripheral institutions. Large matrix elements correspond to strong financial dependence generated through both direct and indirect ownership chains. While this matrix fully characterizes the equilibrium valuation problem, it contains only structural information describing how institutions are connected; it does not reveal how perturbations propagate once failures occur.

The centre panel displays the effective Ising coupling matrix
\[
J_{ij}
=
\frac{\lambda}{4}
(A_{ij}+A_{ji}),
\]
obtained from the QUBO formulation after symmetrization. The mapping preserves the hierarchical organization already present in the financial network while transforming the contagion problem into an interacting spin system suitable for annealing-based optimization. Strong couplings remain concentrated among the highly interconnected core institutions, whereas peripheral institutions interact much more weakly. Consequently, the Ising Hamiltonian faithfully captures the dominant pathways through which financial distress propagates across the network.

The right panel presents the bailout susceptibility matrix
\begin{equation*}
\chi_{ij} = \beta_{\rm eq}\left(\langle\sigma_i^z\sigma_j^z\rangle - \langle\sigma_i^z\rangle\langle\sigma_j^z\rangle\right),
\end{equation*}
evaluated numerically from equilibrium Metropolis Monte Carlo samples via  Eq.~(\ref{eq:bailout_susceptibility}). Unlike the previous two matrices, the susceptibility is not an input to the model but an emergent equilibrium response function. Each matrix element quantifies how strongly the failure tendency of institution $i$ is reduced by strengthening the bailout field acting on institution $j$. Thus, $\chi$ measures the propagation of stabilizing interventions through the collective dynamics of the financial system.

Several features are immediately apparent. The prominent upper-left block indicates that interventions applied to the core institutions produce the strongest collective response, demonstrating that these institutions form a tightly coupled dynamical subsystem. By contrast, the weaker lower-right block shows that interventions directed solely toward peripheral institutions generate comparatively little systemic benefit. The off-diagonal blocks reveal that stabilizing selected core institutions also reduces the failure tendency of institutions outside the core, illustrating how bailout benefits propagate beyond the directly rescued institutions.

This figure also provides the physical justification for the susceptibility-driven bailout strategy developed in the previous section. The aggregate susceptibility,
\[
\mu_j
=
\frac{1}{n}
\sum_i
\chi_{ij},
\]
measures the total stabilizing influence of institution $j$ on the entire financial network. Institutions with large values of $\mu_j$ produce the greatest reduction in systemic risk per intervention and therefore naturally emerge as preferred bailout candidates. Unlike conventional centrality measures that depend only on network topology, the susceptibility incorporates both the interaction structure and the equilibrium collective dynamics of the system. It therefore represents a genuinely dynamical measure of systemic importance.

Taken together, Fig.~\ref{fig:susceptibility_matrices} summarizes the central conceptual contribution of this work. The dependency matrix describes the financial architecture, the Ising coupling matrix translates that architecture into an optimization model, and the susceptibility matrix reveals how interventions propagate through the resulting collective dynamics. This progression,
\[
A
\rightarrow
J
\rightarrow
\chi,
\]
provides a unified framework linking financial networks, statistical mechanics, and optimization-based systemic risk management. Figure~\ref{fig:susceptibility_smoothened} presents a smoothened, continuous rendering of this susceptibility landscape, further highlighting the concentration of systemic response around the core institutions.

\begin{figure*}[p]
    \centering
    \includegraphics[width=0.85\textwidth]{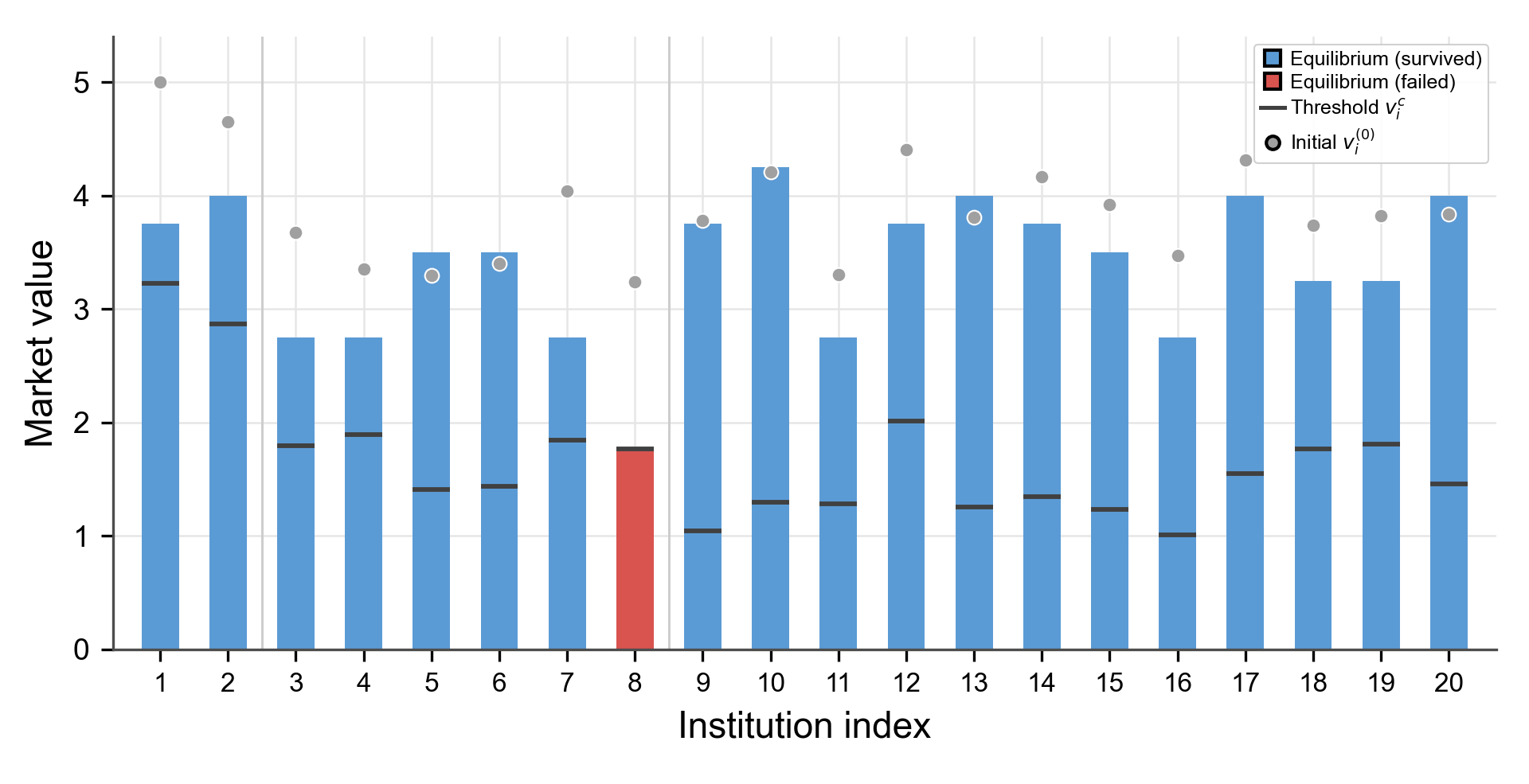}
    \caption{
Equilibrium market valuation for a synthetic financial network with $n=20$ institutions and $m=5$ primitive assets. For each institution, the initial market value $v_i^{(0)}$, the critical solvency threshold $v_i^{c}$, and the equilibrium market value $v_i^{*}$ obtained from the nonlinear fixed-point equation are shown. Institutions are grouped into mega-bank, regional, and peripheral tiers. Institution~8 is the only institution for which $v_i^{*}<v_i^{c}$, indicating equilibrium failure. The total equilibrium loss is approximately $9.69$ ($12.5\%$ of the initial market value), demonstrating how threshold-induced nonlinearities can generate endogenous failures even under relatively small valuation shocks.
}
    \label{fig:market_equ}
\end{figure*}

\begin{figure*}
    \centering
    \includegraphics[width=1\linewidth]{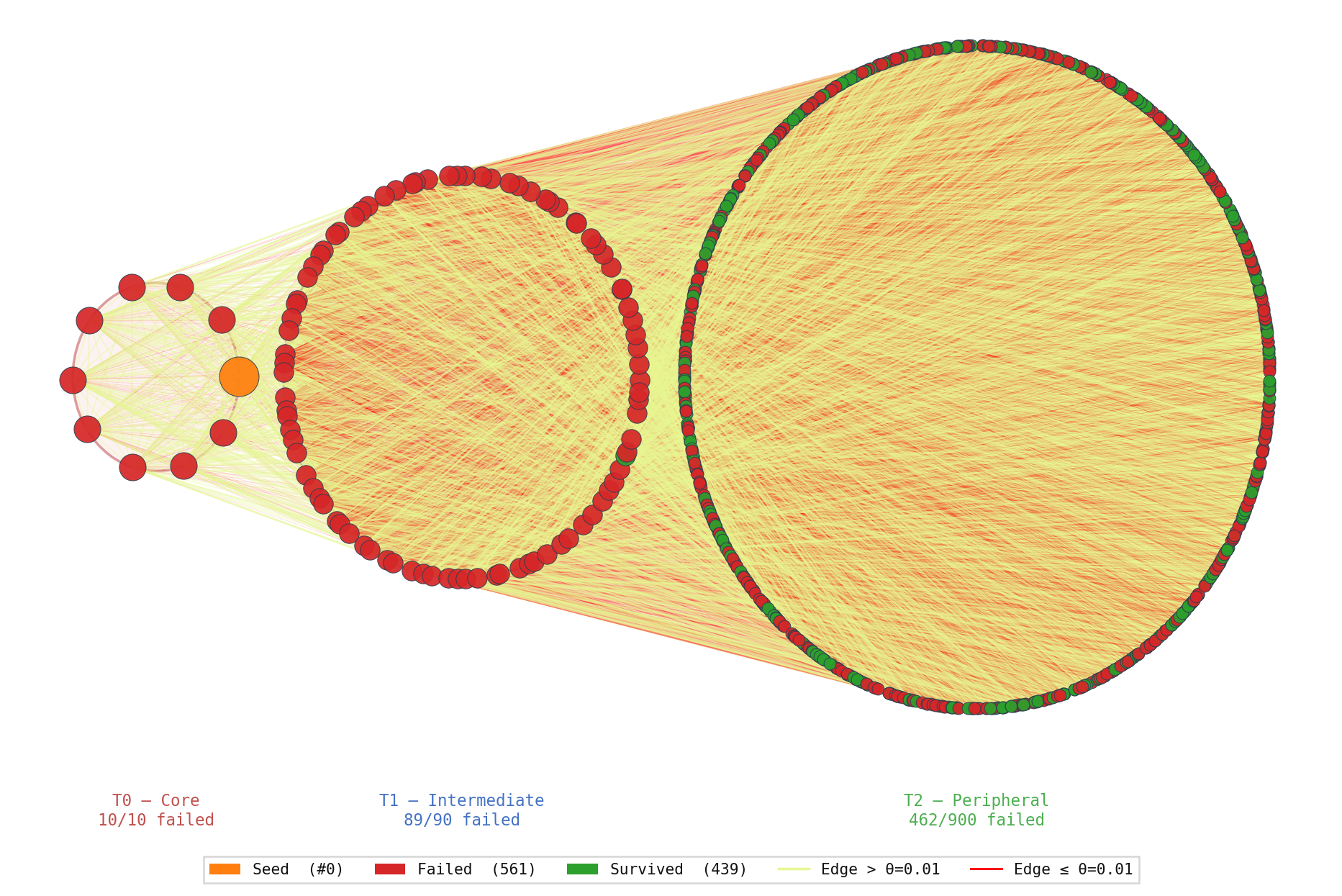}
    \caption{Worst-case cascade obtained by solving the Maximum Cascade Failure Problem for a synthetic financial network containing $n=1000$ institutions using simulated quantum annealing (SQA). The orange node denotes the initially shocked seed institution. Red nodes represent institutions belonging to the optimized cascade, while green nodes remain solvent. Edges correspond to dependency links with strength exceeding the contagion threshold $\theta$. The optimized cascade contains $561$ failed institutions at a minimum QUBO energy of $-498.50$, including all $10$ mega-banks, $89$ regional institutions, and $462$ of the $900$ peripheral institutions. The result illustrates the emergence of a large, self-consistent systemic failure cluster generated from a localized initial shock.}
    \label{fig:maxcasc}
\end{figure*}

\begin{figure*}
    \centering
    \includegraphics[width=0.80\textwidth]{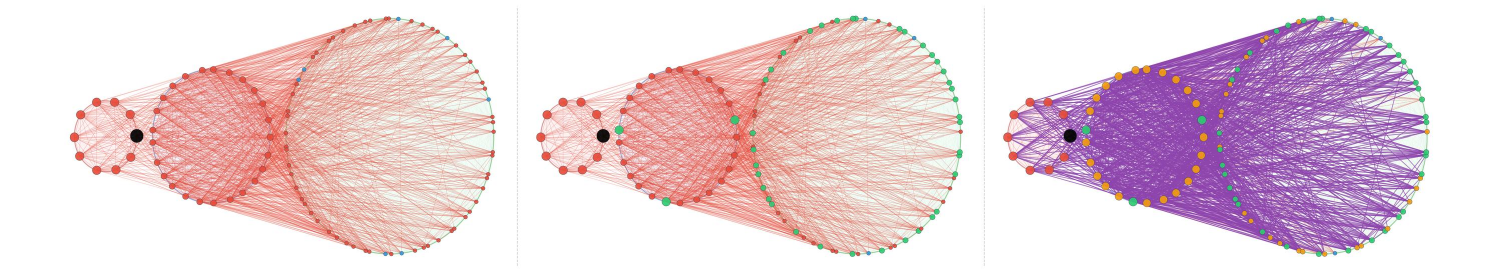}
    \vspace{0.1 cm}
    \includegraphics[width=0.85\textwidth]{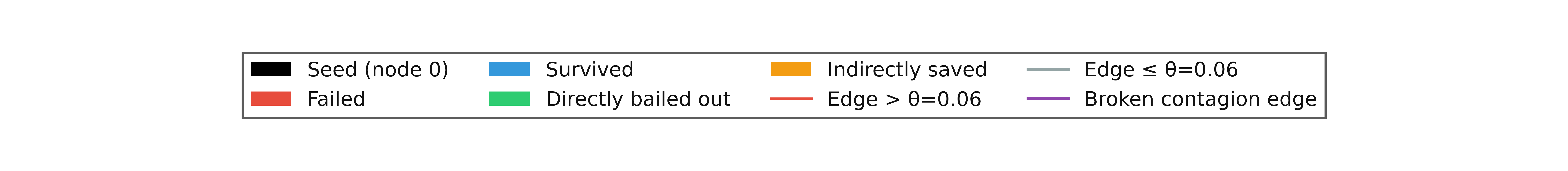}
    \caption{
    Joint-QUBO bailout optimization for a hierarchical financial network containing $100$ institutions organized into core ($10$ nodes), intermediate ($25$ nodes), and peripheral ($65$ nodes) layers. \textit{Left:} equilibrium cascade before intervention ($m=+0.840$). \textit{Centre:} optimal bailout configuration obtained under the prescribed budget constraint. \textit{Right:} equilibrium after intervention. The optimized bailout blocks contagion between network layers, producing numerous indirectly protected institutions (orange) and reducing the system magnetization to $m=-0.060$.
    }
    \label{fig:bailout_100}
\end{figure*}

\begin{figure*}
    \centering
    \includegraphics[width=0.8\textwidth]{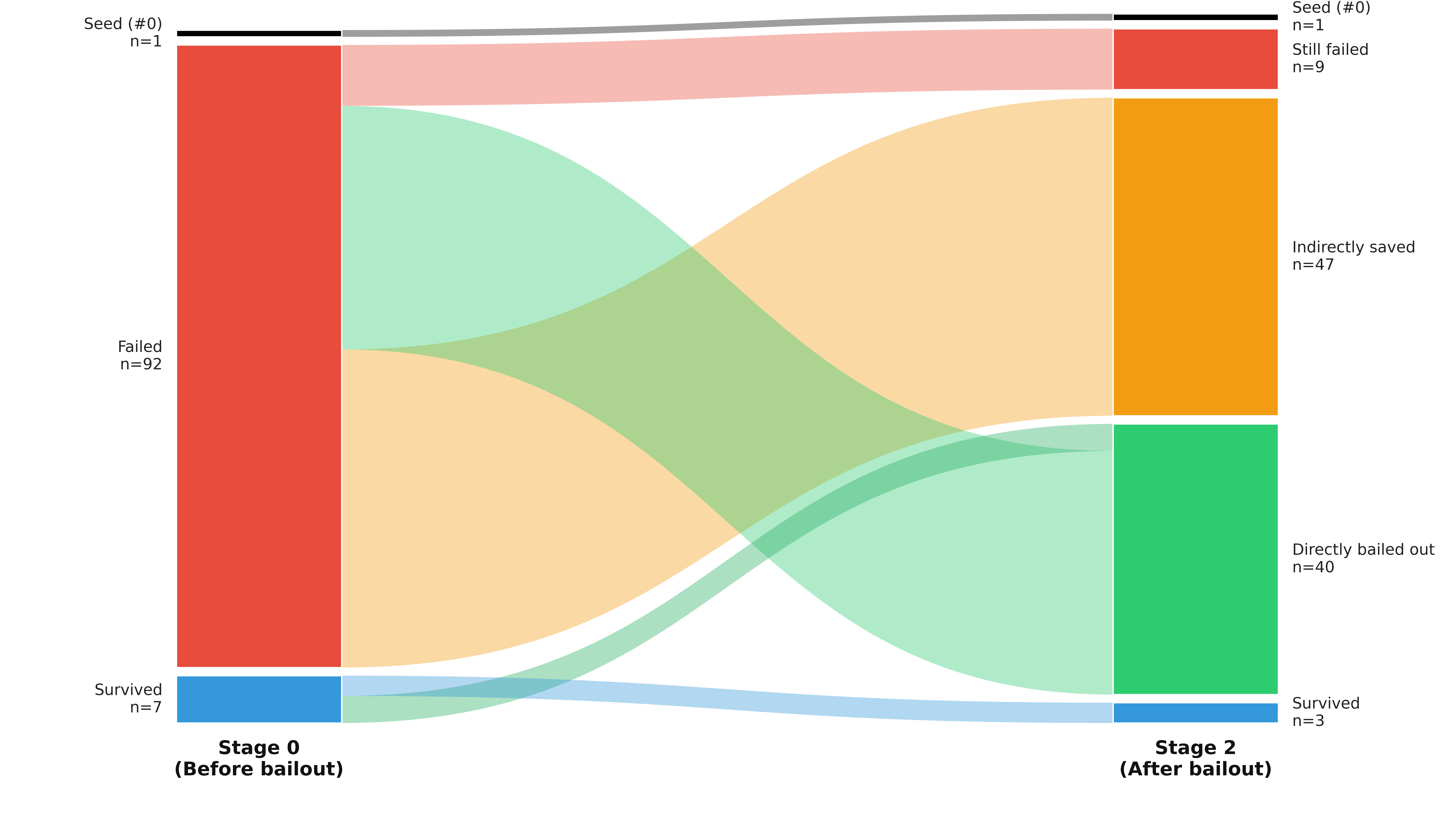}
    \caption{
Node-state transition induced by the optimal bailout strategy for a financial network containing $100$ institutions. The Sankey diagram tracks the evolution of institution states from the equilibrium cascade before intervention (left) to the equilibrium obtained after solving the joint QUBO bailout optimization problem (right). Initially, the network contains one externally shocked seed institution, $92$ failed institutions, and $7$ surviving institutions. Following the optimized intervention, $40$ institutions receive direct bailouts, leading to the indirect stabilization of an additional $47$ institutions through the suppression of contagion pathways. Only $9$ institutions remain failed after intervention, while the initially shocked seed remains failed by construction. The figure illustrates that the optimized bailout strategy protects a substantially larger fraction of the financial system than the number of institutions directly rescued, demonstrating the collective nature of systemic stabilization achieved through strategically targeted interventions.
}
    \label{fig:sankey}
\end{figure*}

\begin{figure*}
    \centering
    \includegraphics[width=0.8\linewidth]{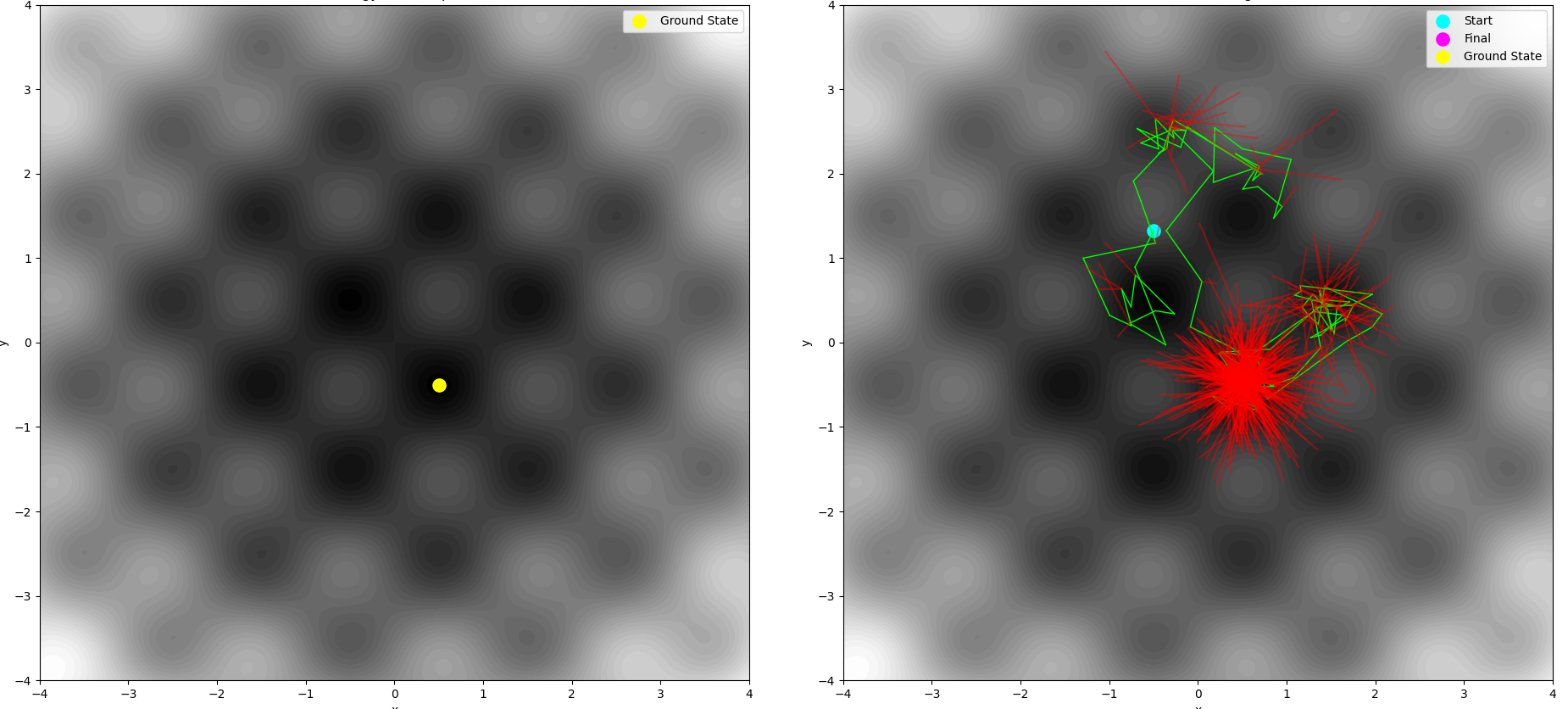}
    \caption{
Illustration of simulated annealing on a representative rugged energy landscape. \textit{Left:} Two-dimensional energy landscape containing multiple local minima separated by energy barriers. The yellow marker denotes the global minimum (ground state). \textit{Right:} Representative annealing trajectory generated using the Metropolis algorithm. The cyan marker indicates the initial random configuration, while the magenta marker denotes the final configuration obtained after annealing. Green trajectory segments correspond to accepted Monte Carlo moves, whereas red segments represent rejected trial moves. At high temperatures, occasional uphill moves enable the system to escape local minima and explore the energy landscape. As the temperature decreases, the acceptance of energetically unfavorable moves becomes increasingly suppressed, causing the trajectory to converge toward the low-energy basin surrounding the global minimum. This figure illustrates the optimization principle underlying the annealing-based solution of the Ising--QUBO formulations developed in this work.
}
    \label{fig:simanneal}
\end{figure*}

\begin{figure*}[p]
    \centering
    \includegraphics[width=1\textwidth]{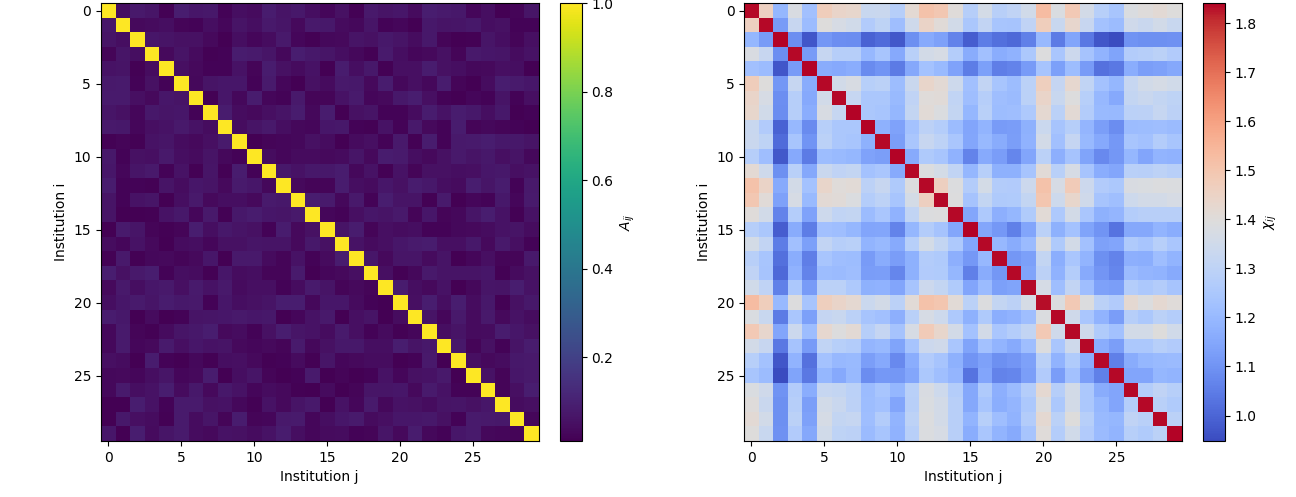}
    \caption{
    Progression from the financial network representation to the equilibrium response of the Ising model for a synthetic $30$-institution core - intermediate - periphery financial network. \textit{Left:} dependency matrix $A=\hat C(I-C)^{-1}$ describing the propagation of market value through direct and indirect ownership relations. \textit{Right:} equilibrium bailout susceptibility matrix,
    $\chi=\beta_{\mathrm{eq}}\mathrm{Cov}(s)$,
    computed from equilibrium Metropolis Monte Carlo simulations at $\beta_{\mathrm{eq}}=1.66$. The pronounced response within the core block demonstrates that interventions applied to highly interconnected institutions generate the strongest system-wide stabilization, while weaker responses in the peripheral block indicate limited global influence. The susceptibility therefore provides a dynamical measure of systemic importance and motivates the susceptibility-driven bailout strategy developed in this work.
    }
    \label{fig:susceptibility_matrices}
\end{figure*}

\begin{figure*}
    \centering
    \includegraphics[width=1\linewidth]{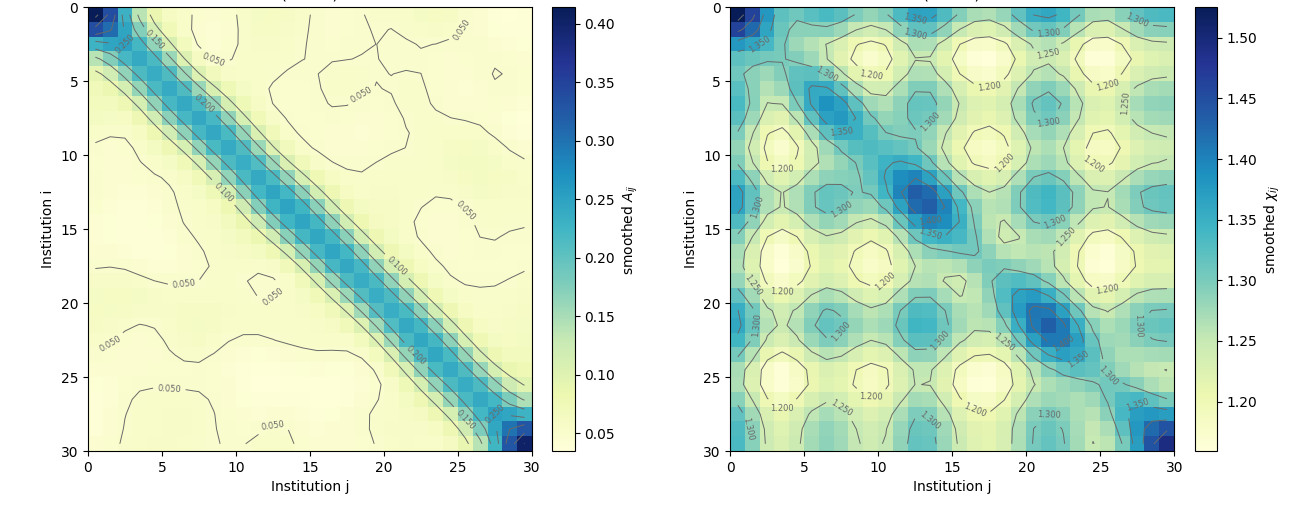}
    \caption{Smoothened rendering of the equilibrium bailout susceptibility
matrix $\chi_{ij}$, interpolated across institution indices. Shows a continuous distribution of susceptibility among different tiers for a given bailout field.}

    \label{fig:susceptibility_smoothened}
\end{figure*}

\section{Conclusions}

In this work, we developed a unified theoretical and computational framework for analysing systemic risk in interconnected financial systems using the language of Ising models and Quadratic Unconstrained Binary Optimization (QUBO). Starting from the financial network model of Elliott, Golub, and Jackson, we showed that equilibrium valuation, cascading financial contagion, worst-case cascade analysis, and optimal regulatory intervention can all be formulated within a common optimization framework. This unification establishes a direct connection between systemic risk analysis, statistical mechanics, and modern optimization techniques, including both classical and quantum annealing.

The starting point of the framework is the equilibrium valuation of financial institutions possessing both primitive assets and cross-holdings in other institutions. The resulting dependency matrix,
\[
A=\hat{C}(I-C)^{-1},
\]

provides a compact description of how value propagates through direct and indirect ownership chains. Introducing solvency thresholds and failure penalties transforms the linear valuation problem into a nonlinear fixed-point problem whose solutions exhibit cascading contagion and amplification of financial distress. This formulation provides a transparent mathematical description of how relatively small external shocks may propagate through ownership networks to generate systemic failures.

Building upon this nonlinear contagion model, we formulated the Maximum Cascade Failure Problem (MCFP), which seeks the largest self-consistent failure configuration compatible with the network dynamics. We further showed that the MCFP admits a natural QUBO representation, allowing the search for maximum cascades to be performed using Ising optimization methods instead of exhaustive enumeration.

We next considered the complementary regulatory problem of minimizing systemic collapse through limited financial intervention. By introducing bailout variables as control fields acting on the Ising Hamiltonian, the Optimal Bailout Allocation Problem was formulated as a controlled interacting spin system. The original bi-level optimization was subsequently transformed into a single joint QUBO involving both failure and bailout variables, enabling the simultaneous determination of equilibrium failure configurations and optimal intervention strategies under budget constraints. Numerical simulations on synthetic financial networks containing up to one thousand institutions demonstrated that strategically selected bailouts substantially suppress cascade propagation while requiring intervention only at a relatively small subset of institutions.

A further contribution of this work is the introduction of bailout susceptibility as a dynamical measure of systemic importance. Unlike conventional centrality measures that depend solely on network topology, the susceptibility matrix quantifies how stabilizing one institution influences the equilibrium failure tendencies of all others through collective interactions. This response function provides a physical interpretation of the optimization results and naturally motivates the susceptibility-driven greedy bailout strategy developed in this work. The progression from the dependency matrix to the effective Ising couplings and finally to the susceptibility matrix establishes a unified picture connecting financial architecture, collective dynamics, and optimal intervention.

Beyond the specific optimization problems considered here, the present framework illustrates how concepts originating in statistical mechanics can provide powerful tools for understanding complex financial systems. Equilibrium valuation determines the underlying financial architecture, nonlinear contagion captures the emergence of systemic crises, Ising optimization identifies both destructive and stabilizing configurations, and susceptibility analysis explains the propagation of regulatory interventions. These elements are not independent models but successive layers of a single theoretical framework for systemic risk analysis.

More broadly, the Ising representation establishes a general bridge between threshold financial networks and interacting spin systems. This connection enables the application of a wide range of theoretical and computational methods from condensed matter physics to systemic risk analysis, extending well beyond optimization alone.

Several directions remain open for future investigation. The present model may be extended to incorporate heterogeneous recovery rates, stochastic asset dynamics, time-dependent contagion processes, multilayer financial networks, and empirical ownership data. From the computational perspective, the QUBO formulations developed here provide a natural platform for hybrid classical--quantum optimization algorithms, variational quantum methods, and emerging quantum annealing hardware. As these technologies continue to mature, the framework presented in this work offers a promising route toward scalable stress testing, real-time cascade prediction, and optimization-based regulatory planning for increasingly interconnected financial systems.

\section{Acknowledgements}

The authors acknowledge support from the National Quantum Mission (NQM), an initiative of the Department of Science and Technology (DST), Government of India under the Project titled \emph{Q-LAT Anneal A - General-Purpose QUBO Compiler for Quantum Many-Body Physics and
Hybrid Optimisation}.
L.N. and A.T. were supported by the DAE Research Visitor Fellowship at The Institute of Mathematical Sciences (IMSc), Chennai. We also thank IMSc for providing HPC resources.
The numerical optimization results reported in this paper were obtained using the in-house \textsc{QAnneal} framework developed by the authors.

\bibliography{references}
\bibliographystyle{apsrev4-2}

\newpage

\appendix

\section{Joint QUBO Coefficients for Optimal Bailout Allocation}
\label{app:qb_derivation}

Here we make explicit the matrix $Q(b)$ appearing in the joint QUBO objective of
Eq.~(49). Substituting the spin--binary correspondence $s_i = 2x_i - 1$ into the
controlled Hamiltonian $H_{\rm ctrl}(s,b)$ of Eq.~(41) and collecting terms into the
quadratic form $x^{T}Q(b)\,x$ (dropping the resulting constant offset) gives the
coefficients
\begin{equation}
Q_{ij}(b) = -2J_{ij}, \qquad i \neq j,
\label{eq:qb_offdiag}
\end{equation}
\begin{equation}
Q_{ii}(b) = 2\sum_{j\neq i} J_{ij} \;-\; 2h_i \;+\; 2\Omega_i b_i.
\label{eq:qb_diag}
\end{equation}

Unlike the Maximum Cascade Failure Problem objective of Eq.~(32), $Q(b)$ contains no
free reward term: its diagonal is built entirely from the physically motivated
vulnerability field $h_i$, the contagion couplings $J_{ij}$, and the bailout field
$\Omega_i b_i$. Because these terms compete rather than uniformly favouring failure,
the ground state of $x^{T}Q(b)\,x$ represents a genuine equilibrium balance between
contagion and stabilization, rather than an unconstrained maximization of failures.

\section{Numerical Implementation Parameters}
\label{app:parameters}

All results reported in Section~\ref{sec:results_discussion} were obtained using the
QAnneal framework~\cite{qanneal2026}. Table~\ref{tab:mcfp_params} lists the parameters
used for the Maximum Cascade Failure Problem (Sec.~\ref{sec:mcfp-ising}),
Table~\ref{tab:obap_params} lists the parameters used for the Optimal Bailout
Allocation Problem (Sec.~\ref{sec:obap-ising}), and Table~\ref{tab:eq_params} lists the
parameters used for the equilibrium valuation experiment of
Sec.~\ref{sec:results_discussion}B.

Throughout this work the reward weight was fixed at $w=0$; the reported maximum
cascades are therefore the forced closure of the seed institution under the
contagion constraints of Eq.~(26), rather than the outcome of an explicit
size-maximizing reward term.

\begin{table}[h]
\centering
\caption{MCFP solver parameters.}
\label{tab:mcfp_params}
\begin{tabular}{lccc}
\hline\hline
Parameter & mcSA & mcSQA & mcSQAPT \\
\hline
Reads / samples          & 1000  & 30       & 10       \\
Sweeps                   & 10000 & --       & --       \\
Trotter slices           & --    & 32       & 16       \\
Replicas                 & --    & --       & 4        \\
PT steps                 & --    & --       & 25       \\
Swap interval            & --    & --       & 1        \\
Mode                     & --    & balanced & balanced \\
$\theta$                 & 0.05  & 0.01     & 0.05     \\
$\lambda$                & 50.0  & 100.0    & 100.0    \\
$w$                      & 0.0   & 0.0      & 0.0      \\
$M$                      & 500.0 & 500.0    & 500.0    \\
\hline\hline
\end{tabular}
\end{table}

\begin{table}[H]
\centering
\caption{OBAP solver parameters.}
\label{tab:obap_params}
\begin{tabular}{lcc}
\hline\hline
Parameter & obapSA & obapSQA \\
\hline
Reads / samples         & 1000  & 10       \\
Sweeps                  & 10000 & --       \\
Trotter slices          & --    & 32       \\
Mode                    & --    & balanced \\
$\theta$                & 0.01  & 0.01     \\
$\lambda$               & 50.0  & 50.0     \\
$M$                     & 500.0 & 500.0    \\
$\lambda_{\rm budget}$  & 200.0 & 200.0    \\
$K$                     & --    & 1000     \\
Seed institution        & 0     & 0        \\
\hline\hline
\end{tabular}
\end{table}

\begin{table}[H]
\centering
\caption{Equilibrium valuation solver parameters.}
\label{tab:eq_params}
\begin{tabular}{lc}
\hline\hline
Parameter & Value \\
\hline
$n$ (institutions)            & 20    \\
$m$ (primitive assets)        & 5     \\
$q$ (binary expansion order)  & 2     \\
Reads / samples               & 5000  \\
Sweeps                        & 50000 \\
\hline\hline
\end{tabular}
\end{table}

\end{document}